\documentclass{amsart}
\long\def\symbolfootnote[#1]#2{\begingroup\def\thefootnote{\fnsymbol{footnote}}
\footnote[#1]{#2}\endgroup}
\newtheorem{thm}{Theorem}[section]
\newtheorem{cor}[thm]{Corollary}
\newtheorem{lem}[thm]{Lemma}
\newtheorem{prop}[thm]{Proposition}
\theoremstyle{definition}
\newtheorem{defn}[thm]{Definition}
\theoremstyle{remark}
\newtheorem{rem}[thm]{Remark}

\begin{document}

\title[Submanifolds in manifolds with metric mixed 3-structures]{Submanifolds in manifolds with metric mixed 3-structures}

\author[S. Ianu\c{s}, L. Ornea, G.E. V\^{\i}lcu]{Stere Ianu\c{s}$^*$, Liviu Ornea, Gabriel Eduard V\^{\i}lcu}\thanks{$^*$Passed away on April 8, 2010}

\date{}
\maketitle

\abstract Mixed 3-structures are odd-dimensional analogues of
paraquaternionic structures. They appear naturally on lightlike
hypersurfaces of almost paraquaternionic hermitian manifolds. We
study invariant and anti-invariant submanifolds in a manifold
endowed with a mixed 3-structure and a compatible (semi-Riemannian)
metric. Particular attention is given to two cases of ambient space:
mixed 3-Sasakian and mixed 3-cosymplectic. \\[1mm]
{\em AMS Mathematics Subject Classification:} 53C15, 53C50, 53C40, 53C12.\\
{\em Key Words and Phrases:} invariant submanifold, anti-invariant
submanifold, mixed 3-structure, Einstein manifold.

\endabstract

\section{Introduction}

 The counterpart in odd dimension of a paraquaternionic
structure was introduced in \cite{IMV}. It is called mixed
3-structure, and appears in a natural way on lightlike hypersurfaces
in almost paraquaternionic hermitian manifolds. Such hypersurfaces
inherit two almost paracontact structures and an almost contact
structure, satisfying analogous conditions to those satisfied by
almost contact 3-structures \cite{KUO}. This concept has been
refined in \cite{CP}, where the authors have introduced positive and
negative  metric mixed 3-structures. The differential geometry of
the semi-Riemannian hypersurfaces of co-index both 0 and 1 in such
manifolds has been recently investigated in \cite{IV2}. In the
present paper, we discuss non-degenerate invariant and
anti-invariant submanifolds in manifolds endowed with metric mixed
3-structures, the relevant ambients being mixed 3-Sasakian and mixed
3-cosymplectic.

The paper is organized as follows. In Section 2 we recall
definitions and basic properties of manifolds with metric mixed
3-structures. In Section 3 we establish several results concerning
the existence of invariant and anti-invariant submanifolds in a
manifold endowed with a metric mixed 3-structure, tangent or normal
to the structure vector fields. Particularly, we show that an
invariant submanifold is either tangent or normal to \emph{all} the
three structure vector fields. Moreover, we prove that a totally
umbilical submanifold of a mixed 3-Sasakian manifold, tangent to the
structure vector fields, is invariant and totally geodesic. This
section ends with a wide range of examples. In Section 4 we study
the anti-invariant submanifolds in a manifold endowed with a mixed
3-cosymplectic or mixed 3-Sasakian structure, normal to the
structure vector fields. In particular, necessary and sufficient
conditions are provided for the connection in the normal bundle to
be trivial. We also provide an example of an anti-invariant flat
minimal submanifold of $S^{4n+3}_{2n+1}$, normal to the structure
vector fields. Section 5 discusses the distributions which naturally
appear on invariant submanifolds of manifolds endowed with metric
mixed 3-structures, tangent to the structure vector fields.
Moreover, we obtain that a non-degenerate submanifold of a mixed
3-Sasakian manifold tangent to the structure vector fields is
totally geodesic if and only if it is invariant. In the last Section
we investigate the geometry of invariant submanifolds of mixed
3-cosymplectic manifolds, normal to the structure vector fields and
prove that such a submanifold admits a para-hyper-K\"{a}hler
structure.

\section{Preliminaries}

An almost product structure on a smooth manifold $\overline{M}$ is a
tensor field $P$ of type (1,1) on $\overline{M}$, $P\neq\pm Id$,
such that
\[
P^2=Id.
\]
where $Id$ is the identity tensor field of type (1,1) on
$\overline{M}$.

An almost complex structure on a smooth manifold $\overline{M}$ is a
tensor field $J$ of type (1,1) on $\overline{M}$ such that
\[
J^2=-Id.
\]

An almost para-hypercomplex structure on a smooth manifold
$\overline{M}$ is a triple $H=(J_{\alpha})_{\alpha=\overline{1,3}}$,
where $J_1$, $J_2$ are almost product structures on $\overline{M}$
and $J_3$ is an almost complex structure on $\overline{M}$,
satisfying:
\[
         J_1J_2=-J_2J_1=J_3.
\]

A semi-Riemannian metric $\overline{g}$ on $(\overline{M},H)$ is
said to be compatible or adapted to the almost para-hypercomplex
structure $H=(J_{\alpha})_{\alpha=\overline{1,3}}$ if it satisfies:
\[
         \overline{g}(J_1X,J_1 Y)=\overline{g}(J_2X,J_2 Y)=-\overline{g}(J_3X,J_3Y)=-\overline{g}(X,Y)
\]
for all vector fields $X$,$Y$ on $\overline{M}$. Moreover, the
triple $(\overline{M},\overline{g},H)$ is said to be an almost
para-hyperhermitian manifold. If $\lbrace{J_1,J_2,J_3}\rbrace$ are
parallel with respect to the Levi-Civita connection of $\overline{g}$,
then the manifold is called para-hyper-K\"{a}hler. Note that, given a para-hypercomplex structure,
compatible metrics might not exist at all, at least in real dimension $4$, as recently shown in
\cite{DGMY}, using an Inoue surface.

An almost hermitian paraquaternionic manifold is a triple
$(\overline{M},\sigma,\overline{g})$, where $\overline{M}$ is a
smooth manifold, $\sigma$ is a rank 3-subbundle of
$End(T\overline{M})$ which is locally spanned by an almost
para-hypercomplex structure $H=(J_{\alpha})_{\alpha=\overline{1,3}}$
and $\overline{g}$ is a compatible metric with respect to $H$.
Moreover, if the bundle $\sigma$ is preserved by the
Levi-Civita connection of $\overline{g}$, then
$(\overline{M},\sigma,\overline{g})$ is said to be a
paraquaternionic K\"{a}hler manifold \cite{GRM}. The prototype of
paraquaternionic K\"{a}hler manifold is the paraquaternionic
projective space $P^n(\mathbb{B})$ as described by Bla\v{z}i\'{c}
\cite{BLZ}.

A submanifold $M$ of a quaternionic K\"{a}hler manifold
$\overline{M}$ is called  quaternionic (respectively
totally real) if each tangent space of $M$ is carried
into itself (respectively into its orthogonal complement) by each section of
$\sigma$. Several examples of paraquaternionic and totally real
submanifolds of $P^n(\mathbb{B})$ are given in \cite{IMAV,MAR}.

\begin{defn}
        Let $\overline{M}$ be a differentiable manifold equipped with a triple
        $(\varphi,\xi,\eta)$, where $\varphi$ is a field  of endomorphisms
        of the tangent spaces, $\xi$ is a vector field and $\eta$ is a
        1-form on $\overline{M}$. If we have:
\begin{equation}\label{1}
        \varphi^2=\tau(-I+\eta\otimes\xi),\ \ \  \eta(\xi)=1
\end{equation}
        then we say that:

        $(i)$ $(\varphi,\xi,\eta)$ is an almost contact
        structure on $\overline{M}$, if $\tau=1$ (\cite{SAS}).

        $(ii)$ $(\varphi,\xi,\eta)$ is an almost paracontact
        structure on $\overline{M}$, if $\tau=-1$ (\cite{SAT}).
\end{defn}

We remark that many authors also include in the above definition the
conditions that
\begin{equation}\label{2}
\varphi\xi=0,\ \eta\circ\varphi=0,
\end{equation}
although these are deducible from (\ref{1}) (see \cite{BLR}).

\begin{defn}
A mixed 3-structure on a smooth manifold $\overline{M}$ is a triple
of structures $(\varphi_\alpha,\xi_\alpha,\eta_\alpha)$,
$\alpha\in\{1,2,3\}$, which are almost paracontact structures for
$\alpha=1,2$ and almost contact structure for $\alpha=3$, satisfying
the following conditions:
\begin{equation}\label{3}
\eta_\alpha(\xi_\beta)=0,
\end{equation}
\begin{equation}\label{4}
\varphi_\alpha(\xi_\beta)=\tau_\beta\xi_\gamma,\ \ \varphi_\beta(\xi_\alpha)=-\tau_\alpha\xi_\gamma,\\
\end{equation}
\begin{equation}\label{5}
\eta_\alpha\circ\varphi_\beta=-\eta_\beta\circ\varphi_\alpha=
\tau_\gamma\eta_\gamma\,,\\
\end{equation}
\begin{equation}\label{6}
\varphi_\alpha\varphi_\beta-\tau_\alpha\eta_\beta\otimes\xi_\alpha=
-\varphi_\beta\varphi_\alpha+\tau_\beta\eta_\alpha\otimes\xi_\beta=
\tau_\gamma\varphi_\gamma\,,
\end{equation}
where $(\alpha,\beta,\gamma)$ is an even permutation of $(1,2,3)$
and $\tau_1=\tau_2=-\tau_3=-1$.

Moreover, if a manifold $\overline{M}$ with a mixed 3-structure
$(\varphi_\alpha,\xi_\alpha,\eta_\alpha)_{\alpha=\overline{1,3}}$
 admits a semi-Riemannian metric $\overline{g}$ such that:
\begin{equation}\label{7}
\overline{g}(\varphi_\alpha X, \varphi_\alpha Y)=\tau_\alpha
[\overline{g}(X,Y)-\varepsilon_\alpha\eta_\alpha(X)\eta_\alpha(Y)],
\end{equation}
for all $X,Y\in\Gamma(T\overline{M})$ and $\alpha=1,2,3$, where
$\varepsilon_\alpha=\overline{g}(\xi_\alpha,\xi_\alpha)=\pm1$, then we
say that $\overline{M}$ has a metric mixed 3-structure and
$\overline{g}$ is called a compatible metric.
\end{defn}

From (\ref{7}) we obtain
\begin{equation}\label{8}
\eta_\alpha(X)=\varepsilon_\alpha\overline{g}(X,\xi_\alpha),\
\overline{g}(\varphi_\alpha X, Y)=- \overline{g}(X,\varphi_\alpha Y)
\end{equation}
for all $X,Y\in\Gamma(T\overline{M})$ and $\alpha=1,2,3$.

Note that if
$(\overline{M},(\varphi_\alpha,\xi_\alpha,\eta_\alpha)_{\alpha=\overline{1,3}},\overline{g})$
is a manifold with a metric mixed 3-structure then from (\ref{8}) it
follows
\[
\overline{g}(\xi_1,\xi_1)=\overline{g}(\xi_2,\xi_2)=-\overline{g}(\xi_3,\xi_3).
\]

Hence the vector fields $\xi_1$ and $\xi_2$ are both either
space-like or time-like and these force the causal character of the
third vector field $\xi_3$. We may therefore distinguish between
positive and negative metric mixed 3-structures, according as
$\xi_1$ and $\xi_2$ are both space-like, or both time-like vector
fields. Because one can check that, at each point of $\overline{M}$,
 there always exists a pseudo-orthonormal frame field given by
$\{(E_i,\varphi_1 E_i, \varphi_2 E_i, \varphi_3
E_i)_{i=\overline{1,n}}\,, \xi_1, \xi_2, \xi_3\}$ we conclude that
the dimension of the manifold is $4n+3$ and the signature of
$\overline{g}$ is $(2n+1,2n+2)$, where we put first the minus signs,
if the metric mixed 3-structure is positive (\emph{i.e.}
$\varepsilon_1=\varepsilon_2=-\varepsilon_3=1$), or the signature of
$\overline{g}$ is $(2n+2,2n+1)$, if the metric mixed 3-structure is
negative (\emph{i.e.}
$\varepsilon_1=\varepsilon_2=-\varepsilon_3=-1$).

\begin{rem} For the time
being, it is not known wether a mixed 3-structure always admits both
positive and negative compatible semi-Riemannian metrics or not. The
cited result in \cite{DGMY} suggests a negative answer, but we do
not have a proof.
\end{rem}

\begin{defn}
Let
$(\overline{M},(\varphi_\alpha,\xi_\alpha,\eta_\alpha)_{\alpha=\overline{1,3}},\overline{g})$
be a manifold with a metric mixed 3-structure.

$(i)$ If  $(\varphi_1,\xi_1,\eta_1,\overline{g})$,
$(\varphi_2,\xi_2,\eta_2,\overline{g})$ are para-cosymplectic
structures and $(\varphi_3,\xi_3,\eta_3,\overline{g})$ is a
cosymplectic structure, \emph{i.e.} the Levi-Civita connection
$\overline{\nabla}$ of $\overline{g}$ satisfies
\begin{equation}\label{9}
\overline{\nabla}\varphi_\alpha=0
\end{equation}
for all $\alpha\in\{1,2,3\}$, then
$((\varphi_\alpha,\xi_\alpha,\eta_\alpha)_{\alpha=\overline{1,3}},\overline{g})$
is said to be a mixed 3-cosymplectic structure on $\overline{M}$.

$(ii)$ If $(\varphi_1,\xi_1,\eta_1,\overline{g})$,
$(\varphi_2,\xi_2,\eta_2,\overline{g})$ are para-Sasakian structures
and $(\varphi_3,\xi_3,\eta_3,\overline{g})$ is a Sasakian structure,
\emph{i.e.}
\begin{equation}\label{10}
(\overline{\nabla}_X\varphi_\alpha)
Y=\tau_\alpha[g(X,Y)\xi_\alpha-\varepsilon_\alpha\eta_\alpha(Y)X]
\end{equation}
for all $X,Y\in\Gamma(T\overline{M})$ and $\alpha\in\{1,2,3\}$, then
$((\varphi_\alpha,\xi_\alpha,\eta_\alpha)_{\alpha=\overline{1,3}},\overline{g})$
is said to be a mixed 3-Sasakian structure on $\overline{M}$.
\end{defn}

Note that from (\ref{9}) it follows:
\begin{equation}\label{11}
\overline{\nabla}\xi_\alpha=0,\, (\text{and hence}\, \overline{\nabla}\eta_\alpha=0),
\end{equation}
and from (\ref{10}) we obtain
\begin{equation}\label{12}
\overline{\nabla}_X\xi_\alpha=-\varepsilon_\alpha\varphi_\alpha X,
\end{equation}
for all $\alpha\in\{1,2,3\}$ and $X\in\Gamma(T\overline{M})$.

Like their Riemannian counterparts,
mixed 3-Sasakian structures are Einstein, but now the scalar curvature can be either positive or negative (see
\cite{CP,IV}):

\begin{thm} Any $(4n+3)-$dimensional manifold endowed with a
mixed $3$-Sasakian structure is an Einstein space with Einstein
constant $\lambda=(4n+2)\varepsilon$, with $\varepsilon=\mp1$,
according as the metric mixed 3-structure is positive or negative,
respectively.
\end{thm}

Several examples of manifolds endowed with metric mixed 3-structures
are given in \cite{IVV,IV2}: $\mathbb{R}^{4n+3}_{2n+1}$ admits a
positive mixed 3-cosymplectic structure, $\mathbb{R}^{4n+3}_{2n+2}$
admits a negative mixed 3-cosymplectic structure, the unit
pseudo-sphere $S^{4n+3}_{2n+1}$ and the real projective space
$P^{4n+3}_{2n+1}(\mathbb{R})$ are the canonical examples of
manifolds with positive mixed 3-Sasakian structures, while the unit
pseudo-sphere $S^{4n+3}_{2n+2}$ and the real projective space
$P^{4n+3}_{2n+2}(\mathbb{R})$ can be endowed with negative mixed
3-Sasakian structures.\\

Let $(\overline{M},\overline{g})$ be a semi-Riemannian manifold and
let $M$ be an immersed submanifold of $\overline{M}$. Then $M$ is
said to be \emph{non-degenerate} if the
restriction of the semi-Riemannian metric $\overline{g}$ to $TM$ is
non-degenerate at each point of $M$. We denote by $g$ the
semi-Riemannian metric induced by $\overline{g}$ on $M$ and by
$TM^\perp$ the normal bundle to $M$. Then we have the following
orthogonal decomposition:
$$T\overline{M}=TM\oplus TM^\perp.$$

Also, we denote by $\overline{\nabla}$ and $\nabla$ the Levi-Civita
connection on $\overline{M}$ and $M$, respectively. Then the Gauss
formula is given by:
\begin{equation}\label{13}
\overline{\nabla}_X Y=\nabla_X Y+h(X,Y)
\end{equation}
for all $X,Y\in\Gamma(TM)$, where
$h:\Gamma(TM)\times\Gamma(TM)\rightarrow\Gamma(TM^\perp)$ is the
second fundamental form of $M$ in $\overline{M}$.

On the other hand, the Weingarten formula is given by:
\begin{equation}\label{14}
\overline{\nabla}_X N=-A_N X+\nabla^\perp_XN
\end{equation}
for any $X\in\Gamma(TM)$ and $N\in\Gamma(TM^\perp)$, where $-A_N X$
is the tangential part of $\overline{\nabla}_X N$ and
$\nabla^\perp_XN$ is the normal part of $\overline{\nabla}_X N$;
$A_N$ and $\nabla^\perp$ are called the shape operator of $M$ with
respect to $N$ and the normal connection, respectively. Moreover,
$h$ and $A_N$ are related by:
\begin{equation}\label{15}
\overline{g}(h(X,Y),N)=g(A_N X,Y)
\end{equation}
for all $X,Y\in\Gamma(TM)$ and $N\in\Gamma(TM^\perp)$.

For the rest of this paper we shall assume that the induced metric is
non-degenerate.

\section{Basic results}

\begin{defn}
A non-degenerate submanifold $M$ of a manifold $\overline{M}$
endowed with a metric mixed 3-structure
$((\varphi_\alpha,\xi_\alpha,\eta_\alpha)_{\alpha=\overline{1,3}},\overline{g})$
is said to be:

$(i)$ \emph{invariant} if $\varphi_\alpha(T_pM)\subset T_pM$, for
all $p\in M$ and $\alpha=1,2,3$;

$(ii)$ \emph{anti-invariant} if $\varphi_\alpha(T_pM)\subset
T_pM^\perp$, for all $p\in M$ and $\alpha=1,2,3$.
\end{defn}

\begin{lem}\label{3.2}
Manifolds with metric mixed 3-structure do not admit anti-invariant
submanifolds tangent to the structure vector fields
$\xi_1,\xi_2,\xi_3$.
\end{lem}
\begin{proof}
If we suppose that $M$ is an anti-invariant submanifold of the
manifold $\overline{M}$ endowed with a metric mixed 3-structure
$((\varphi_\alpha,\xi_\alpha,\eta_\alpha)_{\alpha=\overline{1,3}},\overline{g})$,
tangent to the structure vector fields, then it follows
\[\varphi_\alpha(\xi_\beta)\in T_pM^\perp,\ \alpha\neq\beta.\]
On the other hand, we have from (\ref{4}) that
\[\varphi_\alpha(\xi_\beta)=\tau_\beta\xi_\gamma\in T_pM,\]
for any even permutation $(\alpha,\beta,\gamma)$ of $(1,2,3)$. So
\[\xi_\gamma\in T_pM\cap T_pM^\perp=\{0\},\]
which is a contradiction.
\end{proof}

On the contrary, in mixed 3-Sasakian ambient, a submanifold normal to the structure fields is forced to be anti-invariant:

\begin{lem}\label{3.3}
Let $M$ be a non-degenerate $m$-dimensional submanifold of a
$(4n+3)$-dimensional mixed 3-Sasakian manifold
$((\overline{M},\varphi_\alpha,\xi_\alpha,\eta_\alpha)_{\alpha=\overline{1,3}},\overline{g})$.
If the structure vector fields are normal to $M$, then $M$ is
anti-invariant and $m\leq n$.
\end{lem}
\begin{proof}
By using (\ref{12}) and Weingarten formula, we obtain for all
$X,Y\in\Gamma(TM)$:
\[\overline g(\varphi_\alpha X,Y)=-\varepsilon_\alpha \overline g(\overline{\nabla}_X\xi_\alpha,Y)
=\varepsilon_\alpha g(A_{\xi_\alpha}X,Y)\] and similarly we find
\[\overline g(\varphi_\alpha Y,X)=\varepsilon_\alpha g(A_{\xi_\alpha}Y,X).\]

But since $A_\xi$ is a self-adjoint operator, it follows using also
(\ref{8}) that we have
\[\overline g(\varphi_\alpha X,Y)=0, \forall X,Y\in\Gamma(TM),\ \alpha=1,2,3.\]

Therefore $M$ is anti-invariant  and $m\leq n$ follows.

\end{proof}

\begin{cor}
There do not exist invariant submanifolds in mixed 3-Sasakian
manifolds normal to the structure vector fields. In particular, this is the case for the ambients: $S^{4n+3}_{2n+1}$, $S^{4n+3}_{2n+2}$,
$P^{4n+3}_{2n+1}(\mathbb{R})$ and $P^{4n+3}_{2n+2}(\mathbb{R})$.
\end{cor}

\begin{rem}
Let
$(\overline{M},(\varphi_\alpha,\xi_\alpha,\eta_\alpha)_{\alpha=\overline{1,3}},\overline{g})$
be a manifold endowed with a metric mixed 3-structure and let $M$ be
an anti-invariant submanifold of $\overline{M}$, such that the
structure vector fields are not all normal to the submanifold. Hence
we have $\xi_{\alpha p}^t\neq0$, for $\alpha=1,2$ or $3$, where
$\xi_{\alpha p}^t$ denotes the tangential component of $\xi_{\alpha
p}$, $p\in M$.

We consider the subspaces $\xi_p^t\subset T_pM$, $\xi_p^n\subset
T_pM^\perp$, given by
\[\xi_p^t=Sp\{\xi_{1 p}^t,\xi_{2 p}^t,\xi_{3 p}^t\},\
\xi_p^n=Sp\{\xi_{1 p}^n,\xi_{2 p}^n,\xi_{3 p}^n\},\] where
$\xi_{\alpha p}^n$ denotes the normal component of $\xi_{\alpha p}$,
and let $Q_p$ be the orthogonal complementary subspace to $\xi_p^t$
in $T_pM$, $p\in M$. Therefore we have the decomposition
$T_pM=\xi_p^t\oplus Q_p$.

Now, we put $\mathcal{D}_{i p}=\varphi_i(Q_p)$, $i\in\{1,2,3\}$, and
note that $\mathcal{D}_{1 p},\mathcal{D}_{2 p},\mathcal{D}_{3 p}$
are mutually orthogonal non-degenerate vector subspaces of
$T_pM^\perp$. Moreover, if we let $\mathcal{D}_p=\mathcal{D}_{1
p}\oplus \mathcal{D}_{2 p}\oplus \mathcal{D}_{3 p}$ we note that $\mathcal{D}_p$ and $\xi_p^n$ are also mutually
orthogonal non-degenerate vector subspaces of $T_pM^\perp$. Letting
$\mathcal{D}_p^\perp$ be the orthogonal complementary subspace of
$\xi_p^n\oplus\mathcal{D}_p$ in $T_pM^\perp$, we have the
orthogonal decomposition
$T_pM^\perp=\xi_p^n\oplus\mathcal{D}_p\oplus\mathcal{D}_p^\perp$.
Note that $\mathcal{D}_p^\perp$ is invariant with respect to
$\varphi_i$, $i\in\{1,2,3\}$.

\smallskip

We now prove a rather unexpected result concerning the dimensions of
subspaces $\xi_p^t\subset T_pM$ and $\xi_p^n\subset T_pM^\perp$.
\end{rem}

\begin{prop}
Let
$(\overline{M}^{4n+3},(\varphi_\alpha,\xi_\alpha,\eta_\alpha)_{\alpha=\overline{1,3}},\overline{g})$
be a manifold endowed with a metric mixed 3-structure and let $M$ be
an anti-invariant submanifold of $\overline{M}$, such that the
structure vector fields are not all normal to the submanifold. Then
$\dim \xi_p^t=1$ and $\dim \xi_p^n=2$.
\end{prop}
\begin{proof}
We put $q={\rm dim} \xi_p^t$, $r={\rm dim} \xi_p^n$. If the
dimension of $Q_p$ is $s$, then it is obvious that the dimension of
$\mathcal{D}_p$ is $3s$. On the other hand, since the subspace
$\mathcal{D}_p^\perp$ is invariant with respect to each
$\varphi_\alpha$, it follows that its dimension is $4t$. Taking into
account that we have the decomposition
\[T_p\overline{M}=T_pM\oplus
T_pM^\perp=\xi_p^t\oplus
Q_p\oplus\xi_p^n\oplus\mathcal{D}_p\oplus\mathcal{D}_p^\perp\] we
obtain
\[
4n+3=4t+4s+r+q
\]
and so we deduce that $q+r\equiv\ 3 \mod 4$. In view of Lemma
\ref{3.2} and since $\xi_{\alpha p}^t\neq0$, for $\alpha=1,2$ or
$3$, we have that $q,r\in\{1,2,3\}$ and so we conclude that $(q=1,\
r=2)$ or $(q=2,\ r=1)$.

We distinguish two cases.\\
\emph{Case I.} If $\xi_{\alpha p}^n=0$, then $\xi_{\alpha}$ is
tangent to $M$ and using (\ref{4}) and taking into account that $M$
is anti-invariant, we obtain that $\xi_{\beta}$ and
$\xi_{\gamma}$ are both normal to $M$, where
$\{\alpha,\beta,\gamma\}=\{1,2,3\}$. Therefore we have $q=1$ and
$r=2$.\\
\emph{Case II.} If $\xi_{\alpha p}^n\neq0$, then we prove that is
not possible to have $q=2$ and $r=1$. Indeed, if $r=1$, then
\[\xi_{\beta p}^n=a\xi_{\alpha p}^n,\ \xi_{\gamma p}^n=b\xi_{\alpha p}^n,\]
where $\{\alpha,\beta,\gamma\}=\{1,2,3\}$, and from (\ref{8}) we
obtain
\begin{equation}\label{00001}
g(\varphi_\alpha\xi_{\alpha p}^n,\xi_{\alpha
p}^n)=g(\varphi_\alpha\xi_{\alpha p}^n,\xi_{\beta
p}^n)=g(\varphi_\alpha\xi_{\alpha p}^n,\xi_{\gamma p}^n)=0.
\end{equation}

Since each $\eta_i$ vanishes on $Q_p$, $i\in\{1,2,3\}$, making use
of $(\ref{6})$, $(\ref{7})$ and ($\ref{8}$) we derive for all $X\in
Q_p$:
\begin{equation}\label{00002}
g(\varphi_\alpha\xi_{\alpha p}^n,\varphi_\alpha
X)=g(\varphi_\alpha\xi_{\alpha p}^n,\varphi_\beta
X)=g(\varphi_\alpha\xi_{\alpha p}^n,\varphi_\gamma X)=0.
\end{equation}

On the other hand, since $\mathcal{D}_p^\perp$ is invariant with
respect to $\varphi_\alpha$, we also obtain using ($\ref{8}$) that
we have:
\begin{equation}\label{00003}
g(\varphi_\alpha\xi_{\alpha p}^n,U)=-g(\xi_{\alpha
p}^n,\varphi_\alpha U)=0,
\end{equation}
for all $U\in \mathcal{D}_p^\perp$.

From (\ref{00001}), (\ref{00002}) and (\ref{00003}) we deduce that
$\varphi_\alpha\xi_{\alpha p}^n\in T_pM$. On the other hand, taking
account of (\ref{2}) and since $M$ is anti-invariant, we obtain
\[
\varphi_\alpha\xi_{\alpha p}^n=-\varphi_\alpha\xi_{\alpha p}^t\in
T_pM^\perp.
\]
Therefore it follows that $\varphi_\alpha\xi_{\alpha p}^n=0$ and
using (\ref{1}) we get
\[
0=\varphi^2_\alpha \xi_{\alpha
p}^n=\tau_\alpha[\eta_\alpha(\xi_{\alpha p}^n)\xi_{\alpha
p}^t+(\eta_\alpha(\xi_{\alpha p}^n)-1)\xi_{\alpha p}^n],
\]
which leads to a contradiction: $0=\eta_\alpha(\xi_{\alpha p}^n)=1$.
Therefore it is not possible that $q=2$ and $r=1$.
\end{proof}

\begin{cor}
Let
$(\overline{M}^{4n+3},(\varphi_\alpha,\xi_\alpha,\eta_\alpha)_{\alpha=\overline{1,3}},\overline{g})$
be a manifold endowed with a metric mixed 3-structure and let $M$ be
an anti-invariant submanifold of $\overline{M}$, such that
$\xi_{\alpha p}^t\neq0$, for all $p\in M$ and $\alpha=1,2$ or $3$.
Then it follows that the mapping $ \xi:p\in M\mapsto \xi_p^t\subset
T_pM $ defines a non-degenerate distribution of dimension 1 on $M$.
\end{cor}

In general, an invariant submanifold of a mixed 3-structure is
either tangent or normal to \emph{all} the three structure vector
fields (this is the motivation for the analysis in the last two
sections of the paper):

\begin{prop}
Let
$(\overline{M},(\varphi_\alpha,\xi_\alpha,\eta_\alpha)_{\alpha=\overline{1,3}},\overline{g})$
be a manifold endowed with a metric mixed 3-structure and let $M$ be
an invariant submanifold of $\overline{M}$. Then the structure
vector fields are all either tangent or normal to the submanifold.
\end{prop}
\begin{proof}
We suppose that we have the decomposition:
\begin{equation}\label{new2}
\xi_\alpha=\xi_\alpha^t+\xi_\alpha^n,
\end{equation}
 where
$\xi_\alpha^t$ denotes the tangential component of $\xi_\alpha$ and
$\xi_\alpha^n$ is the normal component of $\xi_\alpha$.

Applying now $\varphi_\alpha$ in (\ref{new2}) and taking account of
(\ref{2}) we obtain:
\[
\varphi_\alpha \xi_\alpha^n=-\varphi_\alpha
\xi_\alpha^t\in\Gamma(TM),
\]
since $M$ is an invariant submanifold of $\overline{M}$.

On the other hand, we derive from (\ref{8}) that we have for all
$X\in\Gamma(TM)$:
\[
g(\varphi_\alpha
\xi_\alpha^n,X)=-\overline{g}(\xi_\alpha^n,\varphi_\alpha X)=0.
\]

Therefore we deduce that $\varphi_\alpha \xi_\alpha^n=0$ and so
$\varphi_\alpha \xi_\alpha^t=0$. Using now (\ref{1}) and
(\ref{new2}) we find
\begin{eqnarray}
0=\varphi_\alpha^2
\xi_\alpha^t&=&\tau_\alpha[-\xi_\alpha^t+\eta_\alpha(\xi_\alpha^t)\xi_\alpha]\nonumber\\
&=&\tau_\alpha[(\eta_\alpha(\xi_\alpha^t)-1)\xi_\alpha^t+\eta_\alpha(\xi_\alpha^t)\xi_\alpha^n].\nonumber
\end{eqnarray}

Consequently, if $\xi_\alpha^t\neq0$ and $\xi_\alpha^n\neq0$, we
obtain a contradiction equating the tangential and normal components
in the above relation. Hence we deduce that $\xi_\alpha$ is either
tangent or normal to the submanifold. Finally, it is obvious that if
one of the structure vector fields is tangent to the submanifold,
then from (\ref{4}) it follows that the next two structure vector
fields are also tangent to the submanifold, because the tangent
space of an invariant submanifold is closed under the action of
$(\varphi_\alpha)_{\alpha=\overline{1,3}}$.
\end{proof}

As in the Riemannian case, we have:

\begin{prop}
Let $(\overline{M},(\varphi_\alpha,\xi_\alpha,\eta_\alpha)_{\alpha=\overline{1,3}},\overline{g})$ be a mixed 3-cosymplectic or mixed 3-Sasakian manifold and let
$M$ be a totally umbilical submanifold tangent to the structure vector fields. Then $M$ is totally geodesic.
\end{prop}
\begin{proof}
If $\overline{M}$ is a mixed 3-cosymplectic manifold, then from
Gauss formula and (\ref{11}) we obtain:
\[
0=\overline{\nabla}_X\xi_\alpha=\nabla_X\xi_\alpha+h(X,\xi_\alpha)
\]
for all $X\in\Gamma(TM)$ and $\alpha=1,2,3$. Therefore, equating the
normal components we find:
\begin{equation}\label{16}
h(X,\xi_\alpha)=0.
\end{equation}

If $\overline{M}$ is a mixed 3-Sasakian manifold, then from Gauss
formula and (\ref{12}) we similarly obtain:
\[
-\varepsilon_\alpha\varphi_\alpha
X=\overline{\nabla}_X\xi_\alpha=\nabla_X\xi_\alpha+h(X,\xi_\alpha).
\]
Taking $X=\xi_\alpha$ in the above equality and using (\ref{2}) we
derive:
\[
0=\nabla_{\xi_\alpha}\xi_\alpha +h(\xi_\alpha,\xi_\alpha)
\]
and so we get:
\begin{equation}\label{17}
h(\xi_\alpha,\xi_\alpha)=0.
\end{equation}

On the other hand, since $M$ is totally umbilical, its
second fundamental form satisfies:
\begin{equation}\label{18}
h(X,Y)=g(X,Y)H
\end{equation}
for all $X,Y\in\Gamma(TM)$, where $H$ is the mean curvature vector
field on $M$.

Taking $X=Y=\xi_\alpha$ in (\ref{18}) and using (\ref{16}) - if the
manifold $\overline{M}$ is mixed 3-cosymplectic, or (\ref{17}) - if
the manifold $\overline{M}$ is mixed 3-Sasakian, we obtain
\[0=\varepsilon_\alpha H\]
and therefore $H=0$. Using again (\ref{18}) we obtain the assertion.
\end{proof}

\begin{cor}\label{3.7}
A totally geodesic submanifold of a mixed 3-Sasakian manifold,
tangent to the structure vector fields, is invariant.
\end{cor}
\begin{proof}
From  (\ref{12}) we obtain that
\[
\varphi_\alpha X=-\varepsilon_\alpha\nabla_X\xi_\alpha\in\Gamma(TM),\
\forall X\in\Gamma(TM)
\]
and the conclusion follows.
\end{proof}

\subsection{Examples}

\subsubsection{Images of holomorphic maps}

Let $M$, $M'$ be manifolds endowed with metric mixed 3-structures
$((\varphi_\alpha,\xi_\alpha,\eta_\alpha)_{\alpha=\overline{1,3}},g)$,
$((\varphi'_\alpha,\xi'_\alpha,\eta'_\alpha)_{\alpha=\overline{1,3}},g')$.
We say that a smooth map $f:M\rightarrow N$ is holomorphic if the
equation
\begin{equation}\label{new1}
f_*\circ \varphi_\alpha=\varphi_\alpha'\circ f_*
\end{equation}
holds for all $\alpha\in\{1,2,3\}$.

We remark now that if $f$ is an holomorphic embedding such that the
image of $f$, denoted by $N'=f(M)$, is a non-degenerate submanifold,
then it is an invariant submanifold. Indeed, if we consider
$X_*,Y_*\in\Gamma(TN')$ such that $f_*X=X_*$ and $f_*Y=Y_*$, where
$X,Y\in\Gamma(TM)$, we obtain using (\ref{new1}):
\[
\varphi'_\alpha X_*=\varphi'_\alpha f_*X=f_*( \varphi_\alpha
X)\in\Gamma(TN')
\]
and therefore $N'$ is an invariant submanifold of $M'$.

On the other hand, we can remark that if $M$ is a manifold endowed
with a metric mixed 3-structure
$((\varphi_\alpha,\xi_\alpha,\eta_\alpha)_{\alpha=\overline{1,3}},g)$
and $M'$ is an invariant submanifold of $M$, tangent to the
structure vector fields, then the restriction of
$((\varphi_\alpha,\xi_\alpha,\eta_\alpha)_{\alpha=\overline{1,3}},g)$
to $M'$ is a metric mixed 3-structure and the inclusion map
$i:M'\rightarrow M$ is holomorphic.

\subsubsection{Correspondence between submanifolds of mixed 3-Sasakian manifolds and paraquaternionic K\"ahler manifolds via semi-Riemannian submersions}

Consider the semi-Riemannian submersion
$\pi:S^{4n+3}_{2n+1}\rightarrow P^n(\mathbb{B})$, with totally
geodesic fibres $S^3_1$. It was used by Bla\v{z}i\'{c}
in order to give a natural and geometrically oriented definition of the
paraquaternionic projective space \cite{BLZ}. If
$((\varphi_\alpha,\xi_\alpha,\eta_\alpha)_{\alpha=\overline{1,3}},g)$
is the standard positive mixed 3-Sasakian structure on
$S^{4n+3}_{2n+1}$ (see \cite{IVV}), then the semi-Riemannian metric
$g'$ of $P^n(\mathbb{B})$ is induced by
\[
g'(X',Y')\circ\pi=g(X^h,Y^h),
\]
for all vector fields $X',Y'\in\Gamma(P^n(\mathbb{B}))$, where
$X^h,Y^h$ are the unique horizontal lifts of $X',Y'$ on
$S^{4n+3}_{2n+1}$. Moreover, each canonical local basis
$H=(J_{\alpha})_{\alpha=\overline{1,3}}$ of $P^n(\mathbb{B})$ is
related with structures $(\varphi_\alpha)_{\alpha=\overline{1,3}}$
of $S^{4n+3}_{2n+1}$ by
\[
J_\alpha X'=\pi_*(\varphi_\alpha X^h),
\]
for any $X'\in\Gamma(P^n(\mathbb{B}))$.

Let now $M$ be an  immersed submanifold of $S^{4n+3}_{2n+1}$ and let $N$
be an immersed submanifold of $P^n(\mathbb{B})$ such that
$\pi^{-1}(N)=M$. Then we have that $N$ is a paraquaternionic
(respectively totally real) submanifold of $P^n(\mathbb{B})$ if and
only if $M$ is an invariant (respectively anti-invariant)
submanifold of $S^{4n+3}_{2n+1}$, tangent (respectively normal) to
the structure vector fields.

In particular, if we consider the
canonical paraquaternionic immersion $i:P^m(\mathbb{B})\rightarrow
P^n(\mathbb{B})$, where $m<n$, we obtain that $M=S^{4m+3}_{2m+1}$ is
an invariant totally geodesic submanifold of $S^{4n+3}_{2n+1}$,
tangent to the structure vector fields. Similarly, if we take the
standard totally real immersion $i:P^m_\nu(\mathbb{R})\rightarrow
P^n(\mathbb{B})$, where $m\leq n$ and $\nu\in\{0,...,m\}$, we
conclude that $M=S^m_\nu$ is an anti-invariant totally geodesic
submanifold of $S^{4n+3}_{2n+1}$, normal to the structure vector
fields.

Moreover, it can be proved that if $\pi:M\rightarrow N$ is a
semi-Riemannian submersion from a mixed 3-Sasakian manifold onto a
paraquaternionic K\"{a}hler manifold which commutes with the
structure tensors of type $(1,1)$ (we note that the corresponding
notion in the Riemannian case was studied in \cite{WAT}), and $M'$,
$N'$ are immersed submanifolds of $M$ and $N$ respectively, such
that $\pi^{-1}(N')=M'$, then $M'$ is an invariant (respectively
anti-invariant) submanifold of $M$, tangent (respectively normal) to
the structure vector fields if and only if $N'$ is a
paraquaternionic (respectively totally real) submanifold of $N$.

\subsubsection{Fibre submanifolds of a semi-Riemannian submersion}
Let $\pi$ be a semi-Riemannian submersion from a manifold $M$
endowed with a metric mixed 3-structure
$((\varphi_\alpha,\xi_\alpha,\eta_\alpha)_{\alpha=\overline{1,3}},g)$
onto an almost hermitian paraquaternionic manifold $(N,\sigma,g')$,
which commutes with the structure tensors of type $(1,1)$. The
horizontal and vertical distributions induced by $\pi$ are closed
under the action of $\varphi_\alpha$, $\alpha=1,2,3$, and therefore
we conclude that the fibres are invariant submanifolds of $M$.
Moreover, we have
\[
J_\alpha\pi_*\xi_\alpha=\pi_*\varphi_\alpha\xi_\alpha=0,
\]
for $\alpha=1,2,3,$ and hence we deduce that $\xi_1,\xi_2,\xi_3$ are
vertical vector fields.

In particular, since the semi-Riemannian submersion
$\pi:S^{4n+3}_{2n+1}\rightarrow P^n(\mathbb{B})$ given above
commutes with the structure tensors of type $(1,1)$, we have that
$S^3_1$ is an invariant submanifold of $S^{4n+3}_{2n+1}$, tangent to
the structure vector fields.

\subsubsection{The Clifford torus $S^1(\frac{1}{\sqrt{2}})\times S^1(\frac{1}{\sqrt{2}})\subset S^7_3$} Let $S^7_3$ be the 7-dimensional
unit pseudo-sphere in $\mathbb{R}^8_4$, endowed with standard
positive mixed 3-Sasakian structure
$((\varphi_\alpha,\xi_\alpha,\eta_\alpha)_{\alpha=\overline{1,3}},g)$
(see \cite{IVV}). Let $H=\{J_1,J_2,J_3\}$ be the almost
para-hypercomplex structure of $\mathbb{R}^8_4$ defined by
\[
J_1((x_i)_{i=\overline{1,8}})=
(-x_7,x_8,-x_5,x_6,-x_3,x_4,-x_1,x_2),
\]
\[ J_2((x_i)_{i=\overline{1,8}})=
(x_8,x_7,x_6,x_5,x_4,x_3,x_2,x_1),
\]
\[
J_3((x_i)_{i=\overline{1,8}})=(-x_2,x_1,-x_4,x_3,-
x_6,x_5,-x_8,x_7),
\]
which is compatible with the semi-Riemannian $\overline{g}$ on
$\mathbb{R}^8_4$, given by
\[
        \overline{g}((x_i)_{i=\overline{1,8}},(y_i)_{i=\overline{1,8}})=
        -\sum_{i=1}^{4}x_iy_i+\sum_{i=5}^{8}x_iy_i.
\]

If $S^1(\frac{1}{\sqrt{2}})$ is a circle of radius
$\frac{1}{\sqrt{2}}$, we consider the submanifold
$M=S^1(\frac{1}{\sqrt{2}})\times S^1(\frac{1}{\sqrt{2}})$ of
$S^7_3$. The position vector $X$ of $M$ in $S^7_3$ in
$\mathbb{R}^8_4$ has components given by
\[
N=\frac{1}{\sqrt{2}}(0,0,0,0,\cos u_1,\sin u_1,\cos u_2,\sin u_2),
\]
$u_1$ and $u_2$ being parameters on each $S^1$.

The tangent space is spanned by $\{X_1, X_2\}$, where
\[
X_1=\frac{1}{\sqrt{2}}(0,0,0,0,-\sin u_1,\cos u_1,0,0),
\]
\[
X_2=\frac{1}{\sqrt{2}}(0,0,0,0,0,0,-\sin u_2,\cos u_2)
\]
and the structure vector fields $\xi_1,\xi_2,\xi_3$ of $S^7_3$
restricted to $M$ are given by
\[
\xi_1=\frac{1}{\sqrt{2}}(\cos u_2,-\sin u_2,\cos u_1,-\sin
u_1,0,0,0,0),
\]
\[
\xi_2=\frac{1}{\sqrt{2}}(-\sin u_2,-\cos u_2,-\sin u_1,-\cos
u_1,0,0,0,0),
\]
\[
\xi_3=\frac{1}{\sqrt{2}}(0,0,0,0,\sin u_1,-\cos u_1,\sin u_2,-\cos
u_2).
\]

Since $\varphi_\alpha X$ is the tangent part of $J_\alpha X$, for
all $X\in\Gamma(TM)$ and $\alpha\in\{1,2,3\}$ (see \cite{IVV}), we
obtain:
\[
g(\varphi_\alpha X_i,X_j)=\overline{g}(J_\alpha X_i,X_j)=0,
\]
for all $\alpha\in\{1,2,3\}$ and  $i,j\in\{1,2\}$. Therefore $M$ is
an anti-invariant submanifold of $S^7_3$. On the other hand, it is
easy to verify that
 $\xi_1,\xi_2$ are normal to $M$ and since $\xi_3=-X_1-X_2$, we deduce that $\xi_3$ is
tangent to the submanifold.

\section{Anti-invariant submanifolds of manifolds endowed with metric mixed 3-structures, normal to the structure vector fields}

Let $M$ be an $n$-dimensional anti-invariant submanifold of a
manifold endowed with a metric mixed 3-structure
$(\overline{M},(\varphi_\alpha,\xi_\alpha,\eta_\alpha)_{\alpha=\overline{1,3}},\overline{g})$.
From Lemma \ref{3.2} it follows that the structure vector
fields $\xi_1,\xi_2,\xi_3$ cannot be tangent to $M$, unlike the case
of anti-invariant submanifolds in manifolds endowed with almost
contact structures, where the structure vector field can be both
tangent and normal (see \cite{KIM,KON,LOY,YK2}). Next we suppose
that the structure vector fields are normal to $M$.

Define the distribution
$\xi=\{\xi_1\}\oplus\{\xi_2\}\oplus\{\xi_3\}$ and  set
$\mathcal{D}_{\alpha p}=\varphi_\alpha(T_pM)$, for $p\in M$ and
$\alpha=1,2,3$. We note that $\mathcal{D}_{1 p}$, $\mathcal{D}_{2
p}$, $\mathcal{D}_{3 p}$ are mutually orthogonal non-degenerate
vector subspaces of $T_pM^\perp$. Indeed, by using (\ref{6}) and
(\ref{8}) we obtain
\[
\overline{g}(\varphi_\alpha X, \varphi_\beta Y)=-\overline{g}( X,
\varphi_\alpha\varphi_\beta Y)=-\tau_\gamma\overline{g}( X, \varphi_\gamma
Y)=0
\]
for all $X,Y\in T_pM$, where $(\alpha,\beta,\gamma)$ is an even
permutation of $(1,2,3)$.

Moreover, the subspaces
\[
\mathcal{D}_p=\mathcal{D}_{1 p}\oplus\mathcal{D}_{2
p}\oplus\mathcal{D}_{3 p},\ p\in M
\]
define a non-trivial subbundle of dimension $3n$ on $TM^\perp$. Note
that $\mathcal{D}$ and $\xi$ are mutually orthogonal
subbundle of $TM^\perp$ and let $\mathcal{D}^\perp$ be the
orthogonal complementary vector subbundle of $\mathcal{D}\oplus\xi$
in $TM^\perp$. So we have the orthogonal decomposition:
\[TM^\perp=\mathcal{D}\oplus\mathcal{D}^\perp\oplus\xi.\]

\begin{lem}
$(i)$ $\varphi_\alpha \mathcal{D}_{\alpha p}\subset T_pM$, $\forall p\in M$, $\alpha=1,2,3.$

$(ii)$ $\varphi_\alpha \mathcal{D}_{\beta p}\subset \mathcal{D}_{\gamma p}$, $\forall p\in M$, $\alpha=1,2,3.$

$(iii)$ The subbundle $\mathcal{D}^\perp$ is invariant under the action
of $\varphi_\alpha$, $\alpha=1,2,3$.

$(iv)$ $\varphi^2_\alpha (TM^\perp)\subset TM^\perp$,
$\forall\alpha=1,2,3.$
\end{lem}
\begin{proof}
$(iv)$ is a consequence of the first three claims. $(i)$ and $(ii)$ follow, respectively, from \eqref{1} and \eqref{6}. It remains to prove $(iii)$. If $U\in\Gamma(\mathcal{D}^\perp)$, then using (\ref{2}) and
(\ref{8}) we  obtain
\[
\overline{g}(\varphi_\alpha U,\xi_\alpha)=-\overline{g}
(U,\varphi_\alpha\xi_\alpha)=0,\ \alpha=1,2,3.
\]

Similarly, using (\ref{4}) and (\ref{8}) we get
\[
\overline{g}(\varphi_\alpha U,\xi_\beta)=-\overline{g}
(U,\varphi_\alpha\xi_\beta)=-\tau_\beta\overline{g} (U,\xi_\gamma)=0
\]
for any even permutation $(\alpha,\beta,\gamma)$ of $(1,2,3)$.

On the other hand, if $U\in\Gamma(\mathcal{D}^\perp)$ and
$X\in\Gamma(TM)$, then using (\ref{1}) and (\ref{8}) we obtain:
\[
\overline{g}(\varphi_\alpha U,\varphi_\alpha X)=-\overline{g}
(U,\varphi_\alpha^2 X)=\tau_\alpha\overline{g} (U,X)=0,\ \alpha=1,2,3
\]
and similarly, using (\ref{6}) and (\ref{8}) we have
\[
\overline{g}(\varphi_\alpha U,\varphi_\beta X)=-\overline{g}
(U,\varphi_\alpha\varphi_\beta X)=-\tau_\gamma\overline{g} (U,\varphi_\gamma
X)=0
\]
for any even permutation $(\alpha,\beta,\gamma)$ of $(1,2,3)$. This ends the proof.
\end{proof}

\begin{lem}
If $M$ is an anti-invariant submanifold of a mixed 3-cosymplectic or
mixed 3-Sasakian manifold
$(\overline{M},(\varphi_\alpha,\xi_\alpha,\eta_\alpha)_{\alpha=\overline{1,3}},\overline{g})$,
normal to the structure vector fields, then the distribution $\xi$
on $\overline{M}$ is integrable.
\end{lem}
\begin{proof}
If $\overline{M}$ is a mixed 3-cosymplectic manifold then the
assertion is a direct consequence of (\ref{11}). On the other hand,
if $M$ is a mixed 3-Sasakian manifold, then using (\ref{4}) and
(\ref{12}) we obtain for any
$N\in\Gamma(\mathcal{D}\oplus\mathcal{D}^\perp)$:
\[
\overline{g}([\xi_\alpha,\xi_\beta],N)=(\varepsilon_\beta\tau_\alpha+\varepsilon_\alpha\tau_\beta)g(\xi_\gamma,N)
=0
\]
for any even permutation $(\alpha,\beta,\gamma)$ of $(1,2,3)$.
\end{proof}

\begin{lem}\label{6.3}
If $M$ is an anti-invariant submanifold of a mixed 3-cosymplectic or
mixed 3-Sasakian manifold
$(\overline{M},(\varphi_\alpha,\xi_\alpha,\eta_\alpha)_{\alpha=\overline{1,3}},\overline{g})$,
normal to the structure vector fields, then the following equation holds
good:
\[R^\perp(X,Y)\xi_\alpha=0,\ \forall X,Y\in\Gamma(TM),\ \alpha=1,2,3.\]
\end{lem}
\begin{proof}
From the Weingarten formula we have for any $X\in\Gamma(TM)$ and
$\alpha=1,2,3$:
\begin{equation}\label{25}
\overline{\nabla}_X \xi_\alpha=-A_{\xi_\alpha}
X+\nabla^\perp_X\xi_\alpha.
\end{equation}
If $\overline{M}$ is mixed 3-cosymplectic, then identifying the normal
components in (\ref{11}) and (\ref{25}) we
obtain:
\[
\nabla^\perp_X\xi_\alpha=0, \forall X\in\Gamma(TM),\ \alpha=1,2,3,
\]
and the conclusion follows.

If $\overline{M}$ is mixed 3-Sasakian, from
(\ref{12}) and (\ref{25}) we obtain in a similar way
that
\begin{equation}\label{26}
\nabla^\perp_X\xi_\alpha=-\varepsilon_\alpha\varphi_\alpha X,\ \forall
X\in\Gamma(TM),\ \alpha=1,2,3.
\end{equation}

Using now the Gauss and Weingarten formulas, we get
\begin{equation}\label{26b}
(\overline{\nabla}_X \varphi_\alpha)Y=-A_{\varphi_\alpha
Y}X+\nabla^\perp_X\varphi_\alpha Y-\varphi_\alpha \nabla_XY-\varphi_\alpha
h(X,Y),
\end{equation}
for $X,Y\in\Gamma(TM)$ and $\alpha=1,2,3$.

On the other hand, from (\ref{10}) we obtain
\[
(\overline{\nabla}_X\varphi_\alpha) Y=\tau_\alpha g(X,Y)\xi_\alpha.
\]

Identifying now the normal components in the last two equations we
derive:
\begin{equation}\label{27}
\nabla^\perp_X\varphi_\alpha Y=\tau_\alpha g(X,Y)\xi_\alpha+\varphi_\alpha
\nabla_XY+(\varphi_\alpha h(X,Y))^n
\end{equation}
where  $(\varphi_\alpha h(X,Y))^n$ denotes the normal component of
$\varphi_\alpha h(X,Y)$.

Using now (\ref{26}) and (\ref{27}) we deduce
\[
\nabla^\perp_X\nabla^\perp_Y\xi_\alpha=-\varepsilon_\alpha[\tau_\alpha
g(X,Y)\xi_\alpha+\varphi_\alpha \nabla_XY+(\varphi_\alpha h(X,Y))^n]
\]
and
\[
\nabla^\perp_Y\nabla^\perp_X\xi_\alpha=-\varepsilon_\alpha[\tau_\alpha
g(Y,X)\xi_\alpha+\varphi_\alpha \nabla_YX+(\varphi_\alpha h(Y,X))^n].
\]

Finally we derive:
\begin{eqnarray*}
R^\perp(X,Y)\xi_\alpha&=&\nabla^\perp_X\nabla^\perp_Y\xi_\alpha-\nabla^\perp_Y\nabla^\perp_X\xi_\alpha-\nabla^\perp_{[X,Y]}\xi_\alpha\\
&=&\varepsilon_\alpha(\varphi_\alpha \nabla_YX-\varphi_\alpha
\nabla_XY)+\varepsilon_\alpha\varphi_\alpha[X,Y]\\
&=&0.
\end{eqnarray*}
\end{proof}

We can now prove the main result of this section: the flatness of the normal connection of an
anti-invariant submanifold in a mixed 3-Sasakian or mixed
3-cosymplectic manifold implies strong restrictions on the behavior
of the submanifold (compare with  \cite[Theorem]{YI} for totally real
submanifolds in K\"{a}hler manifolds (where flat normal connection implies flatness of the submanifold) and with
\cite[Proposition 11]{LOY} and  \cite[Corollary 2.1, page 126]{YK2}, for
anti-invariant submanifolds in Sasakian manifolds).

\begin{thm}
Let $M$ be an anti-invariant submanifold of minimal codimension in a
manifold $\overline{M}$ endowed with a metric mixed 3-structure
$((\varphi_\alpha,\xi_\alpha,\eta_\alpha)_{\alpha=\overline{1,3}},\overline{g})$,
such that the structure vector fields are normal to $M$.

$(i)$ If
$(\overline{M},(\varphi_\alpha,\xi_\alpha,\eta_\alpha)_{\alpha=\overline{1,3}},\overline{g})$
is a mixed 3-cosymplectic manifold, then $R^\perp\equiv0$ if and only if $R\equiv0$.

$(ii)$ If
$(\overline{M},(\varphi_\alpha,\xi_\alpha,\eta_\alpha)_{\alpha=\overline{1,3}},\overline{g})$
is a mixed 3-Sasakian manifold, then the connection in the normal
bundle is trivial if and only if $M$ is of constant sectional
curvature $\mp 1$, according as the metric mixed 3-structure
is positive or negative, respectively.
\end{thm}
\begin{proof}
If the dimension of $\overline{M}$ is $(4m+3)$, since the
submanifold $M$ is of minimal codimension, then from Lemma
\ref{3.3} it follows that the dimension of $M$ is $m$ and so
 $\mathcal{D}^\perp=\{0\}$. Therefore we have the orthogonal
decomposition
\[TM^\perp=\mathcal{D}\oplus\xi.\]

On the other hand, identifying the tangential components, from
(\ref{11}) and (\ref{25}) - if $\overline{M}$ is a mixed
3-cosymplectic manifold, or from (\ref{12}) and (\ref{25}) - if
$\overline{M}$ is mixed 3-Sasakian manifold, it follows that
\begin{equation}\label{28}
A_{\xi_\alpha}X=0,\ \forall X\in\Gamma(TM),\ \alpha=1,2,3.
\end{equation}

From (\ref{15}) and (\ref{28}) we can deduce
\[
h(X,Y)\in\Gamma(\mathcal{D}),\ \forall X,Y\in\Gamma(TM)
\]
and so we have
\begin{equation}\label{29}
\varphi_\alpha h(X,Y)\in\Gamma(TM),\ \forall X,Y\in\Gamma(TM),\
\alpha=1,2,3.
\end{equation}
Suppose now that $\overline{M}$ is mixed 3-cosymplectic. Then,
taking into account (\ref{9}) and (\ref{29}) in (\ref{26b}) and
equating the normal components, we obtain
\begin{equation}\label{30}
\nabla^\perp_X\varphi_\alpha Y=\varphi_\alpha \nabla_XY,\ \forall
X,Y\in\Gamma(TM),\ \alpha=1,2,3.
\end{equation}

Using now (\ref{30}) we obtain for all $X,Y,Z\in\Gamma(TM)$ and
$\alpha\in\{1,2,3\}$:
\begin{eqnarray*}
R^\perp(X,Y)\varphi_\alpha Z&=&\nabla^\perp_X\nabla^\perp_Y\varphi_\alpha
Z-
\nabla^\perp_Y\nabla^\perp_X\varphi_\alpha Z-\nabla^\perp_{[X,Y]}\varphi_\alpha Z\\
&=&\nabla^\perp_X(\varphi_\alpha\nabla_Y Z)-\nabla^\perp_Y(\varphi_\alpha\nabla_X Z)-\varphi_\alpha \nabla_{[X,Y]}Z\\
&=&\varphi_\alpha\nabla_X\nabla_Y Z-\varphi_\alpha\nabla_Y\nabla_X Z-\varphi_\alpha \nabla_{[X,Y]}Z\\
&=&\varphi_\alpha R(X,Y)Z
\end{eqnarray*}
and $(i)$ follows from the above equation and
Lemma \ref{6.3}.

For $(ii)$, let $\overline{M}$ be  mixed 3-Sasakian. Then from
(\ref{27}) and (\ref{29}) we deduce that we have for all
$X,Y\in\Gamma(TM)$ and $\alpha\in\{1,2,3\}$:
\begin{equation}\label{31}
\nabla^\perp_X\varphi_\alpha Y=\tau_\alpha g(X,Y)\xi_\alpha+\varphi_\alpha
\nabla_XY.
\end{equation}

Using now (\ref{26}) and (\ref{31}) we derive
\begin{eqnarray*}
R^\perp(X,Y)\varphi_\alpha Z&=&\nabla^\perp_X\nabla^\perp_Y\varphi_\alpha
Z-
\nabla^\perp_Y\nabla^\perp_X\varphi_\alpha Z-\nabla^\perp_{[X,Y]}\varphi_\alpha Z\\
&=&\nabla^\perp_X[\tau_\alpha g(Y,Z)\xi_\alpha+\varphi_\alpha
\nabla_YZ]\\&&-\nabla^\perp_Y[\tau_\alpha
g(X,Z)\xi_\alpha+\varphi_\alpha
\nabla_XZ]\\
&&-[\tau_\alpha g([X,Y],Z)\xi_\alpha+\varphi_\alpha
\nabla_{[X,Y]}Z]\\
&=&\tau_\alpha X g(Y,Z)\xi_\alpha-\varepsilon_\alpha\tau_\alpha
g(Y,Z)\varphi_\alpha X+\tau_\alpha
g(X,\nabla_YZ) \xi_\alpha \\
&&-\tau_\alpha Y g(X,Z)\xi_\alpha+\varepsilon_\alpha\tau_\alpha
g(X,Z)\varphi_\alpha Y-\tau_\alpha
g(Y,\nabla_XZ) \xi_\alpha \\
&&+\varphi_\alpha \nabla_X\nabla_Y Z-\varphi_\alpha \nabla_Y\nabla_X Z\\
&&-\tau_\alpha g([X,Y],Z)\xi_\alpha-\varphi_\alpha \nabla_{[X,Y]}Z)\\
&=&\varphi_\alpha R(X,Y)Z-\varepsilon_\alpha\tau_\alpha[g(Y,Z)\varphi_\alpha
X-g(X,Z)\varphi_\alpha Y]\\
&&+\tau_\alpha[X g(Y,Z)-Y
g(X,Z)+g(X,\nabla_YZ)\\&&-g(Y,\nabla_XZ)-g(\nabla_XY,Z)+g(\nabla_YX,Z)]\xi_\alpha.
\end{eqnarray*}

Therefore, as $\nabla$ is a Riemannian
connection, we deduce
\begin{equation}\label{32}
R^\perp(X,Y)\varphi_\alpha Z=\varphi_\alpha
R(X,Y)Z-\varepsilon_\alpha\tau_\alpha[g(Y,Z)\varphi_\alpha
X-g(X,Z)\varphi_\alpha Y]
\end{equation}
for all $X,Y,Z\in\Gamma(TM)$ and $\alpha\in\{1,2,3\}$.

If the connection of the normal bundle is trivial, \emph{i.e.}
$R^\perp\equiv0$, then from (\ref{32}) we obtain that $M$ has
constant sectional curvature $\varepsilon_\alpha\tau_\alpha$. The
conclusion follows now taking into account that
$\varepsilon_\alpha\tau_\alpha=-1$ if the metric mixed 3-structure
is positive, respectively $\varepsilon_\alpha\tau_\alpha=1$ if the
metric mixed 3-structure
is negative.

Conversely, if $M$ is of constant sectional curvature $\mp 1$,
according as the metric mixed 3-structure
is positive or negative, then from (\ref{32}) we obtain
\[
R^\perp(X,Y)\varphi_\alpha Z=0,\ \forall X,Y,Z\in\Gamma(TM),\
\alpha=1,2,3.
\]

On the other hand,  from Lemma \ref{6.3} we see that the
curvature tensor of the normal bundle annihilates the structure
vector fields. Therefore $R^\perp\equiv0$, \emph{i.e.} the connection in
the normal bundle is trivial.
\end{proof}

\subsection{An example of an anti-invariant submanifold $M$ of minimal codimension in a
mixed 3-Sasakian manifold $\overline{M}$, such that the structure
vector fields are normal to $M$.}\hfill

 Let $H=\{J_1,J_2,J_3\}$ be the almost
para-hypercomplex structure on $\mathbb{R}^{4n+4}_{2n+2}$, given by
\[
J_1((x_i)_{i=\overline{1,4n+4}})=
(-x_{4n+3},x_{4n+4},-x_{4n+1},x_{4n+2},...,-x_3,x_4,-x_1,x_2),
\]
\[ J_2((x_i)_{i=\overline{1,4n+4}})=
(x_{4n+4},x_{4n+3},x_{4n+2},x_{4n+1},...,x_4,x_3,x_2,x_1),
\]
\[
J_3((x_i)_{i=\overline{1,4n+4}})=(-x_2,x_1,-x_4,x_3,...,-
x_{4n+2},x_{4n+1},-x_{4n+4},x_{4n+3}).
\]

It is easily checked that the semi-Riemannian metric
\[
        \overline{g}((x_i)_{i=\overline{1,4n+4}},(y_i)_{i=\overline{1,4n+4}})=
        -\sum_{i=1}^{2n+2}x_iy_i+\sum_{i=2n+3}^{4n+4}x_iy_i
\]
is adapted to the almost para-hypercomplex structure $H$ given
above.

Let $S^{4n+3}_{2n+1}$ be the unit pseudo-sphere  with standard
positive mixed 3-Sasakian structure
$((\varphi_\alpha,\xi_\alpha,\eta_\alpha)_{\alpha=\overline{1,3}},g)$.
This structure is obtained by taking $S^{4n+3}_{2n+1}$ as
hypersurface of $(\mathbb{R}^{4n+4}_{2n+2},\overline{g})$ (see
\cite{IVV}). Let  $T^n$ be the $n$-dimensional real torus
$\underbrace{S^1\times \ldots \times S^1}_{n}$, where $S^1$ is the
unit circle. We can construct a minimal isometric immersion
$f:T^n\rightarrow S^{4n+3}_{2n+1},$ defined by
\[
f(u_1,\ldots,u_n)=\frac{1}{\sqrt{n+1}}(\underbrace{0,\ldots,0}_{2n+2},
\cos x_1, \sin x_1, \ldots,\cos x_n, \sin x_n, \cos x_{n+1}, \sin
x_{n+1}),
\]
where
\[
x_{n+1}=-\sum_{i=1}^{n}x_i,\ u_1=(\cos x_1,\sin x_1),\ldots,
u_n=(\cos x_n,\sin x_n).
\]

The tangent space is spanned by $\{X_1,\ldots,X_n\}$,
where:
\[
X_1=\frac{1}{\sqrt{n+1}}(\underbrace{0,\ldots,0}_{2n+2}, -\sin x_1,
\cos x_1,\underbrace{0, \ldots,0}_{2n-2}, \sin x_{n+1}, -\cos x_{n+1}),
\]
\[
X_2=\frac{1}{\sqrt{n+1}}(\underbrace{0,\ldots,0}_{2n+4}, -\sin x_2,
\cos x_2,\underbrace{0, \ldots,0}_{2n-4}, \sin x_{n+1}, -\cos x_{n+1}),
\]
\[\vdots\]
\[
X_n=\frac{1}{\sqrt{n+1}}(\underbrace{0,\ldots,0}_{4n},-\sin x_n, \cos
x_n, \sin x_{n+1}, -\cos x_{n+1}).
\]
On the other hand, the position vector of $T^n$ in
$\mathbb{R}^{4n+4}_{2n+2}$ has components
\[
N=\frac{1}{\sqrt{n+1}}(\underbrace{0,\ldots,0}_{2n+2}, \cos x_1, \sin
x_1, \ldots,\cos x_n, \sin x_n, \cos x_{n+1}, \sin x_{n+1})
\]
and it is an outward unit spacelike normal vector field of the
pseudo-sphere in $\mathbb{R}^{4n+4}_{2n+2}$. Therefore the structure
vector fields $\xi_1,\xi_2,\xi_3$ of $S^{4n+3}_{2n+1}$ restricted to
$T^n$ are given by
\[
\xi_1=\frac{1}{\sqrt{n+1}}(\cos x_{n+1},-\sin x_{n+1},\cos x_n, -\sin
x_n,\ldots,\cos x_1, -\sin x_1,\underbrace{0,\ldots,0}_{2n+2}),
\]
\[
\xi_2=\frac{1}{\sqrt{n+1}}(-\sin x_{n+1},-\cos x_{n+1},-\sin x_n,-\cos
x_n,\ldots,-\sin x_1,-\cos x_1,\underbrace{0,\ldots,0}_{2n+2}),
\]
\[
\xi_3=\frac{1}{\sqrt{n+1}}(\underbrace{0,\ldots,0}_{2n+2}, \sin
x_1,-\cos x_1, \ldots, \sin x_n,-\cos x_n,  \sin x_{n+1},-\cos x_{n+1}).
\]

Finally, as the structure tensors
$(\varphi_\alpha,\xi_\alpha,\eta_\alpha)_{\alpha=\overline{1,3}}$ of
$S^{4n+3}_{2n+1}$  satisfy
\[
\varphi_\alpha X_i=J_\alpha X_i-\varepsilon_\alpha\eta_\alpha(X_i)N,
\]
and
\[
\eta_\alpha (X_i)=\varepsilon_\alpha\overline{g}(X_i,\xi_\alpha)=0,
\]
for all $i\in\{1,2,\ldots,n\}$ and $\alpha\in\{1,2,3\}$, we conclude
that the immersion $f$ provides a non-trivial example of an
anti-invariant flat minimal submanifold of $S^{4n+3}_{2n+1}$, normal
to the structure vector fields.

\section{Invariant submanifolds of manifolds endowed with metric mixed 3-structures, tangent to the structure vector fields}

Let $(M,g)$ be an invariant submanifold of a manifold endowed with a
metric mixed 3-structure
$(\overline{M},(\varphi_\alpha,\xi_\alpha,\eta_\alpha)_{\alpha=\overline{1,3}},\overline{g})$,
tangent to the structure vector fields $\xi_1,\xi_2,\xi_3$. As above, let $\xi=\{\xi_1\}\oplus\{\xi_2\}\oplus\{\xi_3\}$ and let
$\mathcal{D}$ be the orthogonal complementary distribution to $\xi$
in $TM$. Then we can state the following:
\begin{lem}\label{4.1}
$(i)$ $\varphi_\alpha(T_pM^\perp)\subset T_pM^\perp,\ \forall p\in M,\
\alpha=1,2,3$.

$(ii)$ The distribution $\mathcal{D}$ is invariant
under the action of $\varphi_\alpha$, $\alpha=1,2,3$.
\end{lem}
\begin{proof}
$(i)$ For any $N\in T_pM^\perp$ and $X\in T_pM$, taking account of
(\ref{8}) we obtain:
\[
\overline{g}(\varphi_\alpha N,X)=-g(N,\varphi_\alpha X)=0,
\]
since $M$ is an invariant submanifold.

$(ii)$ For any $X\in\Gamma(\mathcal{D})$, using (\ref{2}) and (\ref{8})
we obtain:
\[
g(\varphi_\alpha X,\xi_\alpha)=-g(X,\varphi_\alpha \xi_\alpha)=0
\]
for $\alpha=1,2,3$.

Similarly, making use of (\ref{4}) and (\ref{8}), we deduce:
\[
g(\varphi_\alpha X,\xi_\beta)=-g(X,\varphi_\alpha \xi_\beta)=-\tau_\beta
g(X,\xi_\gamma)=0
\]
for any even permutation $(\alpha,\beta,\gamma)$ of $(1,2,3)$.
\end{proof}

\begin{prop}\label{4.2}
Let $(M,g)$ be an invariant submanifold of a manifold endowed with a
metric mixed 3-structure
$(\overline{M},(\varphi_\alpha,\xi_\alpha,\eta_\alpha)_{\alpha=\overline{1,3}},\overline{g})$,
such that the structure vector fields $\xi_1,\xi_2,\xi_3$ are
tangent to $M$. If $\overline{M}$ is mixed 3-cosymplectic or mixed
3-Sasakian, then $M$ is mixed 3-cosymplectic and totally
geodesic, respectively mixed 3-Sasakian and totally
geodesic.
\end{prop}
\begin{proof}
Gauss equation implies:
\begin{equation}\label{19}
(\overline{\nabla}_X\varphi_\alpha)Y=(\nabla_X\varphi_\alpha)Y+h(X,\varphi_\alpha
Y)-\varphi_\alpha h(X,Y)
\end{equation}
for all $X,Y\in\Gamma(TM)$.

If $\overline{M}$ is a mixed 3-cosymplectic manifold, then from
(\ref{9}) and (\ref{19}) we deduce:
\[
(\nabla_X\varphi_\alpha)Y+h(X,\varphi_\alpha Y)-\varphi_\alpha h(X,Y)=0
\]
and equating the normal and the tangential components we find
\begin{equation}\label{20}
(\nabla_X\varphi_\alpha)Y=0
\end{equation}
and
\[
h(X,\varphi_\alpha Y)=\varphi_\alpha h(X,Y),\ \alpha=1,2,3.
\]

From (\ref{20}) it follows that the induced metric mixed 3-structure
on $M$ is mixed 3-cosymplectic.

If $\overline{M}$ is a mixed 3-Sasakian manifold, then from
(\ref{10}) and (\ref{19}) we deduce:
\[
(\nabla_X\varphi_\alpha)Y+h(X,\varphi_\alpha Y)-\varphi_\alpha
h(X,Y)=\tau_\alpha[g(X,Y)\xi_\alpha-\varepsilon_\alpha\eta_\alpha(Y)X]
\]
and equating the normal and the tangential components we find
\begin{equation}\label{21}
(\nabla_X\varphi_\alpha)Y=\tau_\alpha[g(X,Y)\xi_\alpha-\varepsilon_\alpha\eta_\alpha(Y)X]
\end{equation}
and
\[
h(X,\varphi_\alpha Y)=\varphi_\alpha h(X,Y),\ \alpha=1,2,3.
\]

From (\ref{21}) it follows that the induced metric mixed 3-structure
on $M$ is mixed 3-Sasakian.

Moreover, making use of (\ref{6}) and (\ref{8}), in both cases we
obtain
\begin{eqnarray*}
    h(X,\varphi_1 Y)&=&\varphi_1h(X,Y)=\tau_1\varphi_2\varphi_3 h(X,Y)=\tau_1\varphi_2
    h(X,\varphi_3Y)\\
    &=&\tau_1\varphi_2h(\varphi_3Y,X)=\tau_1h(\varphi_3Y,\varphi_2X)=\tau_1h(\varphi_2X,\varphi_3Y)\\
    &=&\tau_1\varphi_3h(\varphi_2X,Y)=\tau_1\varphi_3h(Y,\varphi_2X)=\tau_1\varphi_3\varphi_2h(Y,X)\\
    &=&\tau_1\varphi_3\varphi_2h(X,Y)=-\varphi_1 h(X,Y).
\end{eqnarray*}

On the other hand, since $h(X,\varphi_1 Y)=\varphi_1 h(X,Y)$, it follows
that $h(X,Y)=0$, $\forall X,Y\in\Gamma(TM)$ and therefore $M$ is a
totally geodesic submanifold of $\overline{M}$.

\end{proof}

\begin{cor}
An invariant submanifold of a mixed 3-cosymplectic or mixed
3-Sasakian manifold,
tangent to structure vector fields, has
dimension $4k+3$, $k\in\mathbb{N}$. Moreover, the induced metric has
signature $(2k+1,2k+2)$ or $(2k+2,2k+1)$, according to the
metric mixed 3-structure
 being positive or negative.
\end{cor}

\begin{cor}
An invariant submanifold of $\mathbb{R}^{4n+3}_{2n+1}$,
$\mathbb{R}^{4n+3}_{2n+2}$, $S^{4n+3}_{2n+1}$, $S^{4n+3}_{2n+2}$,
$P^{4n+3}_{2n+1}(\mathbb{R})$ and $P^{4n+3}_{2n+2}(\mathbb{R})$,
tangent to the structure vector fields, is locally
isometric with $\mathbb{R}^{4k+3}_{2k+1}$,
$\mathbb{R}^{4k+3}_{2k+2}$, $S^{4k+3}_{2k+1}$, $S^{4k+3}_{2k+2}$,
$P^{4k+3}_{2k+1}(\mathbb{R})$ and $P^{4k+3}_{2k+2}(\mathbb{R})$
respectively, where $0\leq k\leq n$.
\end{cor}

Proposition \ref{4.2} and Corollary \ref{3.7} together imply the
following result, which corresponds to a theorem of Cappelletti
Montano, Di Terlizzi and Tripathi \cite{MTT} for submanifolds in
contact $(\kappa,\mu)$-manifolds.

\begin{prop}
A non-degenerate submanifold of a mixed 3-Sasakian manifold, tangent
to the structure vector fields, is  totally geodesic  if
and only if it is invariant.
\end{prop}

\begin{rem}
The canonical immersions $S^n_\nu\hookrightarrow
S^{4n+3}_{2n+1}$, $S^n_\nu\hookrightarrow S^{4n+3}_{2n+2}$,
$P^n_\nu(\mathbb{R})\hookrightarrow P^{4n+3}_{2n+1}(\mathbb{R})$ and
$P^n_\nu(\mathbb{R})\hookrightarrow P^{4n+3}_{2n+2}(\mathbb{R})$,
where $\nu\in\{0,...,n\}$, provide very natural examples of
anti-invariant totally-geodesic submanifolds, but they are not
tangent to the structure vector fields.
\end{rem}

\begin{lem}
The
distribution $\xi$ of an invariant submanifold of a mixed 3-cosymplectic or
mixed 3-Sasakian manifold tangent to the structure vector fields is integrable.
\end{lem}
\begin{proof}
If $M$ is a mixed 3-cosymplectic manifold, then from (\ref{11}) we
obtain for any $X\in\Gamma(\mathcal{D})$:
\[
\overline{g}([\xi_\alpha,\xi_\beta],X)=
\overline{g}(\overline{\nabla}_{\xi_\alpha}\xi_\beta,X)-\overline{g}(\overline{\nabla}_{\xi_\beta}\xi_\alpha,X)=0.
\]

If $M$ is a mixed 3-Sasakian manifold, then making use of (\ref{4})
and (\ref{12}) we obtain for any $X\in\Gamma(\mathcal{D})$:
\begin{eqnarray*}
\overline{g}([\xi_\alpha,\xi_\beta],X)&=&
\overline{g}(\overline{\nabla}_{\xi_\alpha}\xi_\beta,X)-\overline{g}(\overline{\nabla}_{\xi_\beta}\xi_\alpha,X)\\
&=&-\varepsilon_\beta g(\varphi_\beta\xi_\alpha,X)+\varepsilon_\alpha
g(\varphi_\alpha\xi_\beta,X)\\
&=&(\varepsilon_\beta\tau_\alpha+\varepsilon_\alpha\tau_\beta)g(\xi_\gamma,X)\\
&=&0.
\end{eqnarray*}

Therefore, in both cases it follows that the distribution $\xi$ is
integrable.
\end{proof}

\begin{prop}
Let $(M,g)$ be an invariant submanifold of a manifold $\overline{M}$
endowed with a metric mixed 3-structure, tangent to  the structure
vector fields.

$(i)$ If $\overline{M}$ is  mixed 3-cosymplectic, then the
distribution $\mathcal{D}$ is integrable. Moreover,
the leaves of the foliation are mixed 3-cosymplectic manifold, totally geodesically immersed in $\overline{M}$.

$(ii)$ If $\overline{M}$ is  mixed 3-Sasakian  and $\dim M>3 $, then the distribution
$\mathcal{D}$ is never integrable.
\end{prop}
\begin{proof}
$(i)$ If $\overline{M}$ is mixed 3-cosymplectic, then using
(\ref{11}) we obtain for any $X,Y\in\Gamma(\mathcal{D})$ and
$\alpha=1,2,3$:
\begin{eqnarray*}
\overline{g}([X,Y],\xi_\alpha)&=&
\overline{g}(\overline{\nabla}_XY,\xi_\alpha)-\overline{g}(\overline{\nabla}_YX,\xi_\alpha)\\
&=&-\overline{g}(Y,\overline{\nabla}_X\xi_\alpha)+\overline{g}(X,\overline{\nabla}_Y\xi_\alpha)\\&=&0.
\end{eqnarray*}

Therefore the distribution $\mathcal{D}$ is integrable. Let $M'$ be
a leaf of $\mathcal{D}$. Then for any $X,Y\in\Gamma(TM')$ we have:
\[
\overline{\nabla}_XY=\nabla'_XY+h'(X,Y),
\]
where $\nabla'$ is the connection induced by $\overline{\nabla}$ on
$M'$ and $h'$ is the second fundamental form of the immersion of
$M'$ in $\overline{M}$. Taking into account (\ref{11}) we obtain:
\begin{eqnarray*}
h'(X,\varphi_\alpha Y)&=&\overline{\nabla}_X\varphi_\alpha
Y-\nabla'_X\varphi_\alpha Y\\
&=&(\overline{\nabla}_X\varphi_\alpha)
Y+\varphi_\alpha\overline{\nabla}_XY-\nabla'_X\varphi_\alpha Y\\
&=&\varphi_\alpha\nabla'_XY+\varphi_\alpha h'(X,Y)-\nabla'_X\varphi_\alpha
Y\\
&=&-(\nabla'_X\varphi_\alpha)Y+\varphi_\alpha h'(X,Y).
\end{eqnarray*}

Therefore it follows $(\nabla'_X\varphi_\alpha)Y=0$ and
$h'(X,\varphi_\alpha Y)=\varphi_\alpha h'(X,Y)$, for $\alpha=1,2,3$. From
the last equality we deduce $h'=0$ and the conclusion follows.

$(ii)$ If $\overline{M}$ is a mixed 3-Sasakian manifold, then using
(\ref{8}) and (\ref{12}), we obtain for any
$X,Y\in\Gamma(\mathcal{D})$ and $\alpha=1,2,3$:
\begin{eqnarray*}
\overline{g}([X,Y],\xi_\alpha)&=&
-\overline{g}(Y,\overline{\nabla}_X\xi_\alpha)+\overline{g}(X,\overline{\nabla}_Y\xi_\alpha)\\
&=&\varepsilon_\alpha g(Y,\varphi_\alpha X)-\varepsilon_\alpha
g(X,\varphi_\alpha Y)\\&=&2\varepsilon_\alpha g(Y,\varphi_\alpha X).
\end{eqnarray*}

If we consider now $X$ to be a non-lightlike vector field, then
choosing $Y=\varphi_\alpha X$ in the last identity, we obtain using
(\ref{7}) and (\ref{8}) that we have:
\[
\overline{g}([X,\varphi_\alpha X],\xi_\alpha)=2\varepsilon_\alpha
\tau_\alpha g(X,X)\neq 0.
\]

Therefore the distribution $\mathcal{D}$ is not integrable.
\end{proof}

\section{Invariant submanifolds of manifolds endowed with metric mixed 3-structures, normal to the structure vector fields}

Let $M$ be an invariant submanifold of a manifold endowed with a
metric mixed 3-structure
$(\overline{M},(\varphi_\alpha,\xi_\alpha,\eta_\alpha)_{\alpha=\overline{1,3}},\overline{g})$,
such that the structure vector fields $\xi_1,\xi_2,\xi_3$ are normal
to $M$. We consider $\xi=\{\xi_1\}\oplus\{\xi_2\}\oplus\{\xi_3\}$
and we denote by $\mathcal{D}^\perp$ the orthogonal complementary
subbundle to $\xi$ in $TM^\perp$.

The following result is straightforward:

\begin{lem}
$(i)$ $\varphi_\alpha(T_pM^\perp)\subset T_pM^\perp,\ \forall p\in M,\
\alpha=1,2,3$.

$(ii)$ The subbundle $\mathcal{D}^\perp$ is invariant under the action
of $\varphi_\alpha$, $\alpha=1,2,3$.
\end{lem}

\begin{rem}
If $\overline{M}$ is mixed 3-cosymplectic, then  (\ref{11}) directly implies the integrability of $\xi$ on
$\overline{M}$.
\end{rem}

\begin{prop}\label{5.2}
Let $M$ be an invariant submanifold of a manifold endowed with a
metric mixed 3-structure
$(\overline{M},(\varphi_\alpha,\xi_\alpha,\eta_\alpha)_{\alpha=\overline{1,3}},\overline{g})$,
such that the structure vector fields $\xi_1,\xi_2,\xi_3$ are normal
to $M$. Then $M$ admits an almost para-hyperhermitian structure.
\end{prop}
\begin{proof}
For any $X\in\Gamma(TM)$, we obtain from (\ref{8}) that
\[
\eta_\alpha(X)=\varepsilon_\alpha\overline{g}(X,\xi_\alpha)=0.
\]

Then from (\ref{1}) it follows
\[
\varphi^2_\alpha X=-\tau_\alpha X,\ \alpha=1,2,3
\]
and if we denote by
\[J_\alpha={\varphi_\alpha}_{|M}, \alpha=1,2,3\]
from (\ref{7}) we obtain
\[
J_\alpha J_\beta=-J_\beta J_\alpha=\tau_\gamma J_\gamma,
\]
for any even permutation $(\alpha,\beta,\gamma)$ of $(1,2,3)$.

On the other hand, from (\ref{7}) we get
\[
g(\varphi_\alpha X, \varphi_\alpha Y)=\tau_\alpha g(X,Y),\ \forall
X,Y\in\Gamma(TM),\ \alpha=1,2,3.
\]

Therefore $(M,(J_\alpha)_{\alpha=1,2,3},g)$ is an almost
para-hyperhermitian manifold.
\end{proof}

\begin{cor}
Any invariant submanifold of a manifold endowed with a metric mixed
3-structure, normal to the structure vector fields, has the
dimension $4k$, $k\in\mathbb{N}$, and the induced metric has
signature $(2k,2k)$.
\end{cor}

\begin{prop}
Let $M$ be an invariant submanifold of a manifold endowed with a
metric mixed 3-structure
$(\overline{M},(\varphi_\alpha,\xi_\alpha,\eta_\alpha)_{\alpha=\overline{1,3}},\overline{g})$,
such that the structure vector fields $\xi_1,\xi_2,\xi_3$ are normal
to $M$. If $\overline{M}$ is mixed 3-cosymplectic, then $M$ is a
para-hyper-K\"{a}hler manifold, totally geodesically immersed in
$\overline{M}$.
\end{prop}
\begin{proof}
From Proposition \ref{5.2} it follows that $M$ can be endowed with
an almost para-hypercomplex structure $H=(J_\alpha)_{\alpha=1,2,3}$,
which is para-hyperhermitian with respect to the induced metric $g$.
On the other hand, from (\ref{9}) and Gauss formula we obtain
\[
0=(\overline{\nabla}_X \varphi_\alpha)Y=(\nabla_X
\varphi_\alpha)Y+h(X,\varphi_\alpha Y)-\varphi_\alpha h(X,Y)
\]
for all $X,Y\in\Gamma(TM)$.

From the above identity, equating the normal and tangential
components, it follows that we have:
\begin{equation}\label{23}
h(X,\varphi_\alpha Y)=\varphi_\alpha h(X,Y)
\end{equation}
and
\begin{equation}\label{22}
(\nabla_X J_\alpha)Y=0,
\end{equation}
since $J_\alpha={\varphi_\alpha}_{|M}$.

From (\ref{22}) we deduce that $(M,H=(J_\alpha)_{\alpha=1,2,3},g)$
is a para-hyper-K\"{a}hler manifold and from (\ref{23}) we obtain
similarly as in the proof of Theorem \ref{4.2} that $M$ is totally
geodesic immersed in $\overline{M}$.
\end{proof}

\begin{cor}
The invariant submanifolds of $\mathbb{R}^{4n+3}_{2n+1}$ and
$\mathbb{R}^{4n+3}_{2n+2}$, normal to the structure vector fields,
are locally isometric with $\mathbb{R}^{4k}_{2k}$, where $0\leq
k\leq n$.
\end{cor}

\section*{Acknowledgements} The authors are
partially supported by CNCSIS - UEFISCSU, project PNII - IDEI code
8/2008, contract no. 525/2009.

Stere IANU\c{S} \\
        University of Bucharest,
        Faculty of Mathematics and Computer Science,\\
        Str. Academiei, Nr. 14, Bucharest 70109, Romania\\
        e-mail: ianus@gta.math.unibuc.ro\\

Liviu ORNEA\\
        University of Bucharest,
        Faculty of Mathematics and Computer Science,\\
        Str. Academiei, Nr. 14, Bucharest 70109, Romania\\
        and\\
        Institute of Mathematics ``Simion Stoilow" of the Romanian
        Academy,\\
        Calea Grivi\c{t}ei, Nr. 21, Bucharest 010702, Romania\\
        e-mail: lornea@gta.math.unibuc.ro, liviu.ornea@imar.ro\\

Gabriel Eduard V\^{I}LCU \\
         Petroleum-Gas University of Ploie\c sti,\\
         Department of Mathematics and Computer Science,\\
         Bulevardul Bucure\c sti, Nr. 39, Ploie\c sti 100680, Romania\\
         and\\
        University of Bucharest,
        Faculty of Mathematics and Computer Science,\\
        Research Center in Geometry, Topology and Algebra\\
        Str. Academiei, Nr. 14, Bucharest 70109, Romania\\
         e-mail: gvilcu@mail.upg-ploiesti.ro

\end{document}